\documentclass[12pt,reqno]{amsart}
\usepackage{color}
\usepackage{graphicx}
\usepackage{enumitem}
\usepackage{url}
\usepackage{amssymb}
\usepackage{esint}
\usepackage{xcolor}
\usepackage{mathtools}
\usepackage{comment}
\usepackage{tikz}
\usepackage{slashed}
\usepackage{mathrsfs}

\usepackage[margin = 1in] {geometry}

\usepackage{dsfont}

\newtheorem{theorem}{Theorem}[section]

\newtheorem{proposition}[theorem]{Proposition}
\newtheorem{lemma}[theorem]{Lemma}
\newtheorem{corollary}[theorem]{Corollary}

\theoremstyle{definition}
\newtheorem{example}[theorem]{Example}
\newtheorem{definition}[theorem]{Definition}
\newtheorem*{remark*}{Remark}

\newtheorem{remark}[theorem]{Remark}

\numberwithin{equation}{section}

\usepackage[sorted-cites]{amsrefs} 
\usepackage[colorlinks=true, allcolors=blue]{hyperref} 

\theoremstyle{plain}
\newtheorem{mthm}{Theorem}

\newcommand{\D}{\textrm{d}}
\newcommand{\Sec}{\operatorname{Sec}}  
\newcommand{\Ric}{\operatorname{Ric}}  
\newcommand{\scalM}{\operatorname{scal}_{g_M}} 
\newcommand{\supp}{\operatorname{supp}} 
\newcommand{\Avcw}{ \widehat{A}\textrm{-vcw}_{2} } 
\newcommand{\Af}{ \widehat{\textbf{A}} } 
\newcommand{\ch}{\textrm{ch}} 
\newcommand{\ind}{{\rm ind}} 
\newcommand{\sflow}{{\rm sf}} 

\makeatletter
\@namedef{subjclassname@2020}{\textup{2020} Mathematics Subject Classification}
\makeatother

\title[Bottom spectrum, vertical $\widehat{A}$-cowaist and scalar curvature rigidity]{Bottom spectrum, vertical $\widehat{A}$-cowaist and scalar curvature rigidity}

\author[Daoqiang Liu]{Daoqiang Liu}

\address{Chern Institute of Mathematics \& LPMC, Nankai University, Tianjin 300071, China}
\email{\href{mailto:dqliu@nankai.edu.cn}{dqliu@nankai.edu.cn}}
\urladdr{\href{https://www.dqliu.cn}{www.dqliu.cn}}

\begin{document}



\begin{abstract}
    We introduce the vertical \(\widehat{A}\)-cowaist, a codimension-one
    invariant for partitioned manifolds. It extends the concept of infinite
    vertical \(\widehat{A}\)-cowaist for bands to arbitrary partitioned
    manifolds, which may be noncompact and have compact boundary.
    We establish a sharp inequality relating the scalar curvature, the bottom
    spectrum of the Laplacian, and this invariant.
    As an application, we obtain a high-dimensional analogue of
    Munteanu-Wang's bottom spectrum estimate.
    We also prove a quantitative strengthening of Anghel's theorem together
    with a boundary version, as well as a Calabi-Yau type theorem that goes
    beyond the dimensional restrictions of the earlier \(\mu\)-bubble method.
    Our approach is based on deformed Dirac operators.
\end{abstract}

\maketitle





\section{Introduction}\label{sec:intro}

The study of scalar curvature is a central topic in differential geometry,
closely linked to topology, index theory, and global analysis. A key question
is how scalar curvature bounds interact with large-scale invariants on
complete manifolds --- particularly on noncompact ones, where topology and
spectrum combine in subtle ways.

A particularly powerful obstruction to positive scalar curvature comes from
the Dirac operator. Lichnerowicz~\cite{Lic63} showed that a closed spin
manifold with positive scalar curvature must have vanishing
$\widehat{A}$-genus. Gromov and Lawson~\cite{GL83} then introduced
enlargeability and used index-theoretic methods to obtain far-reaching
obstructions. These results reveal a deep connection between scalar curvature
and topology; for a comprehensive overview, we refer to Gromov's \emph{Four
lectures on scalar curvature}~\cite{Gro23}.

A complementary direction concerns the relationship between scalar curvature and spectral invariants. For a complete Riemannian manifold $(M,g_M)$, the
Laplacian $\Delta$ is self-adjoint (see~\cite{Gaffney54}). The spectrum $\sigma(M):=\sigma(-\Delta)$ of $M$ is a closed subset of $[0,\infty)$. The bottom spectrum is defined by
\[
\lambda_1(M,g_M):=\inf \sigma(M),
\]
and it admits the variational characterization
\begin{equation}\label{defn:bottom_spectrum}
    \lambda_1(M,g_M)=\inf_{u\in C_c^{\infty}(M)\setminus\{0\}}
    \frac{\int_M |{\rm d} u|^2_{g_M}\,\mathrm{d}V_{g_M}}{\int_M u^2\,\mathrm{d}V_{g_M}},
\end{equation}
where $C_c^{\infty}(M)$ denotes the space of compactly supported smooth
functions.  The quantity $\lambda_1(M,g_M)$ captures fundamental large-scale
geometric information, and understanding its interplay with scalar curvature
remains a subtle and active problem, particularly on noncompact manifolds.


\subsection{A sharp geometric inequality}

A classical theorem of Anghel~\cite{An93} states that a noncompact spin manifold that can be partitioned by a closed hypersurface with nonzero \(\widehat{A}\)-genus 
does not admit a complete metric of positive scalar curvature. In other words, the topological obstruction to positive scalar curvature is already detected by the separating hypersurface. Our \emph{vertical \(\widehat{A}\)-cowaist} extends the idea of Cecchini-Zeidler~\cite{CZ24} --- who introduced the concept of infinite vertical \(\widehat{A}\)-cowaist to study distance estimates in the setting of odd-dimensional bands --- to arbitrary partitioned manifolds, in order to quantify such obstructions.

For the precise definition, let \(N \subset M\) be a closed hypersurface partitioning \(M\), so that \(M = M_+ \cup M_-\) and \(N = M_+ \cap M_-\). A Hermitian vector bundle \(\mathcal{E}\) over \(M\) is called \textit{compatible} if it is a trivial bundle with trivial connection at infinity (i.e., outside a compact subset) and near the boundary. A smooth map \(\rho_M: M\to U(l)\) with values in the unitary group $U(l)$ is called \textit{compatible} if it is trivial, i.e., locally constant at infinity and near the boundary.

\begin{itemize}
\item If \(m\) is odd, a compatible bundle \(\mathcal{E}\) over \(M\) is called a \emph{vertical \(\widehat{A}\)-admissible bundle} for the partition \((M,N)\) if
\[
\int_N \Af(N) \wedge \ch(\mathcal{E}|_N) \neq 0.
\]

\item If \(m\) is even, a compatible bundle \(\mathcal{E}\) over \(M\) together with a compatible map \(\rho_M \in C^\infty(M,U(l))\) is called a \emph{vertical \(\widehat{A}\)-admissible pair} \((\mathcal{E},\rho_M)\) for the partition \((M,N)\) if
\[
\int_N \Af(N) \wedge \ch(\mathcal{E}|_N) \wedge \ch(\rho_M|_N) \neq 0,
\]
where \(\ch(\rho_M)\) denotes the odd Chern character form (cf.~\cite{Ge93}, see also \cite{Zh01}*{(1.50)}).
\end{itemize}

For a Hermitian bundle \(\mathcal{E}\), denote by \(R^{\mathcal{E}}\) the curvature tensor of the connection on \(\mathcal{E}\).
Its norm is defined as
\[
\| R^{\mathcal{E}} \|_{\infty}
:= \sup_{p\in M} \sup_{ \substack{v_1,v_2\in T_p M\\ |v_1|=|v_2|=1} } |R^{\mathcal{E}}(v_1\wedge v_2)|,
\]
where \(|R^{\mathcal{E}}(v_1\wedge v_2)|\) denotes the operator norm of the endomorphism \(R^{\mathcal{E}}(v_1\wedge v_2)\).
For a smooth map \(\rho_M\in C^{\infty}(M,U(l))\), the associated curvature is
\begin{equation}\label{eq:defn_rho_curvature}
   R^{\rho_M}:=\frac{1}{4}(\rho_M^{-1}(\D\rho_M))^2, 
\end{equation}
and its norm \(\|R^{\rho_M}\|_{\infty}\) is defined analogously.

\begin{definition}
The \textit{vertical \(\widehat{A}\)-cowaist} of the pair \((M,N)\) is
\[
 \Avcw(M,N):=
\begin{cases}
\displaystyle \Big(\inf_{\mathcal{E}} \|R^{\mathcal{E}}\|_{\infty}\Big)^{-1}, & \text{if } m \text{ is odd}, \\[14pt]
\displaystyle \Big(\inf_{(\mathcal{E},\rho_M)} \bigl\{ \|R^{\mathcal{E}}\|_{\infty} + \|R^{\rho_M}\|_{\infty} \bigr\}\Big)^{-1}, & \text{if } m \text{ is even},
\end{cases}   
\]
where \(\mathcal{E}\) ranges over all vertical \(\widehat{A}\)-admissible bundles for \((M,N)\) and \((\mathcal{E},\rho_M)\) ranges over all vertical \(\widehat{A}\)-admissible pairs for \((M,N)\).
\end{definition}

By pulling back a suitable Hermitian vector bundle $\mathcal{E}_0$ (and, when \(m\) is even, a suitable unitary-valued map \(\rho_0\)) from the sphere \(\mathbf{S}^{m-1}\) via a degree one map $\Theta$, we obtain an \(\widehat{A}\)-admissible bundle $\Theta^*\mathcal{E}_0$ 
(respectively, an \(\widehat{A}\)-admissible pair $(\Theta^*\mathcal{E}_0,\Theta^*\rho_0)$)\footnote{Here
\(\Theta^*\mathcal{E}_0\) is called \(\widehat{A}\)-admissible if
\(\int_M \Af(M)\wedge \ch(\Theta^*\mathcal{E}_0)\neq 0\); the pair
\((\Theta^*\mathcal{E}_0,\Theta^*\rho_0)\) is called \(\widehat{A}\)-admissible
if \(\int_M \Af(M)\wedge \ch(\Theta^*\mathcal{E}_0) \wedge
\ch(\Theta^*\rho_0) \neq 0\).} over \(N\).  
Via the projection \(p:U\to N\) of a small tubular neighborhood \(U=N\times(-\delta,\delta)\), where \(\delta>0\), of \(N\), this bundle (or pair) is pulled back to \(U\).
Using appropriate cut-off functions, it extends to a vertical
\(\widehat{A}\)-admissible bundle (or pair) over the whole manifold \(M\).
Hence the set of such objects over \(M\) is nonempty, and by construction
\(\Avcw(M,N) > 0\).

\begin{definition}
    The \textit{vertical \(\widehat{A}\)-cowaist} of \(M\) itself is
    \[
        \Avcw(M) := \sup_{N} \Avcw(M,N),
    \]
    the supremum taken over all closed partitioning hypersurfaces
    \(N \subset M\).
\end{definition}

This definition generalizes the one in \cite{CZ24}*{Definition~7.3}, which is
itself an analogue of Gromov's \(K\)-cowaist~\cite{Gro96}. The simplest example
of a manifold with infinite vertical \(\widehat{A}\)-cowaist is one that can be
partitioned by a closed hypersurface \(N\) with \(\widehat{A}(N) \neq 0\); for
further examples, we refer to~\cite{CZ24}.

Recall the notion of ends. Let \(M\) be a smooth manifold and fix an exhaustion
\(K_1 \subseteq K_2 \subseteq \cdots\) by compact submanifolds. Set
\(U_i = M \setminus K_i\). An \emph{end} of \(M\) is a sequence
\(\{V_i\}_{i=1}^\infty\) where each \(V_i\) is a connected component of
\(U_i\) and \(V_1 \supseteq V_2 \supseteq \cdots\); the number of ends is
independent of the exhaustion.

Manifolds with ends appear naturally in geometric analysis (see,
e.g.,~\cite{CG72}, \cite{Cai91}, \cite{LT87}). For our purposes, the key fact is that a
noncompact manifold with at least two ends can always be separated by a compact
hypersurface (see~\cite{CRZ23}*{Lemma~2.4}). Hence the vertical
\(\widehat{A}\)-cowaist is well defined under the assumptions of the first main
theorem.

\begin{mthm}\label{mthm:codimension_one}
    Let $(M,g_M)$ be an $m$-dimensional complete noncompact Riemannian spin manifold with at least two ends. 
    Then
    \begin{equation}\label{eq:codimension_one_formula}
        \inf_M \scalM + \frac{4m}{m-1}\, \lambda_1(M,g_M) \leq \frac{2m(m-1)}{\Avcw(M)} .
    \end{equation}
\end{mthm}

The following example exhibits a family of warped product metrics for which
equality holds in \eqref{eq:codimension_one_formula}, showing that the
inequality is sharp and that rigidity fails in general.

\begin{example}\label{example:mthm_codimension_one}
Let \(N\) be a closed Ricci-flat spin manifold of dimension \(m-1\) with \(\widehat{A}(N)\neq 0\) (for instance, a \(K3\) surface when \(m=5\)). Consider the warped product manifold \(M = N \times \mathbf{R}\) equipped with the metric
\[
g_M = \cosh^{\frac{2}{a}}(at)\, g_N + dt^{2},
\]
where \(a \in \bigl[\frac{m-1}{2}, \frac{m}{2}\bigr]\) (cf.~\cite{MWang24}*{Example~1.5}). Its scalar curvature is
\[
\scalM = -m(m-1) + (m-1)(m-2a)\cosh^{-2}(at) \in \bigl[-m(m-1),\, -2a(m-1)\bigr],
\]
and the bottom spectrum of the Laplacian satisfies \(\lambda_1(M,g_M) = \frac{(m-1)^{2}}{4}\).
Since \(\widehat{A}(N)\neq 0\), we have \(\Avcw(M) = \infty\).

As \(\cosh^{-2}(at) \to 0\) when \(t \to \pm\infty\), we obtain
\(\inf_M \scalM = -m(m-1)\) for every \(a\) in this interval;
for \(a = \frac{m}{2}\) the scalar curvature is constant and the
infimum is attained.
Thus,
\[
\inf_M \scalM + \frac{4m}{m-1}\,\lambda_1(M,g_M)
    = -m(m-1) + m(m-1) = 0,
\]
which shows that equality holds in \eqref{eq:codimension_one_formula}.
Consequently, the coefficient \(\frac{4m}{m-1}\) in front of
\(\lambda_1(M,g_M)\) cannot be improved.
\end{example}

For an \(m\)-dimensional simply connected complete manifold with sectional
curvature \(\Sec \le \kappa < 0\), McKean~\cite{McKean70} proved the lower
bound \(\lambda_1 \ge -\frac{(m-1)^2}{4}\kappa\). If an \(m\)-dimensional
complete manifold satisfies \(\Ric \ge (m-1)\kappa\) with \(\kappa \le 0\),
then Cheng~\cite{Ch75} established the upper bound
\[
\lambda_1 \le -\frac{(m-1)^2}{4}\kappa .
\]

In the case of \(3\)-manifolds with a (non-positive) scalar curvature lower bound, Munteanu-Wang \cite{MWang23}, \cite{MWang24} obtained sharp bottom
spectrum estimates, assuming additionally a Ricci curvature lower bound and
topological conditions coming from the level-set method for Green's
functions.

Theorem~\ref{mthm:codimension_one} relates scalar curvature to spectral
geometry on complete manifolds. As a direct consequence, we obtain an
analogue of Cheng's estimate in which the Ricci curvature lower bound is
replaced by a scalar curvature lower bound that may even be positive.

\begin{theorem}\label{thm:bottom_spectrum_all_dimension}
Let $(M,g_M)$ be an $m$-dimensional complete noncompact Riemannian spin manifold with at least two ends, such that
$\scalM \ge \kappa$ on $M$ for some constant $\kappa \le \dfrac{2m(m-1)}{\Avcw(M)}$.
Then the bottom spectrum satisfies
\[
    \lambda_1(M,g_M) \le \frac{m-1}{4m} \bigg( \frac{2m(m-1)}{\Avcw(M)} - \kappa  \bigg).
\]
\end{theorem}

A noncompact manifold that can be partitioned by a closed hypersurface
with nonzero \(\widehat{A}\)-genus automatically has at least two ends and
infinite vertical \(\widehat{A}\)-cowaist. A particularly interesting special
case of Theorem~\ref{thm:bottom_spectrum_all_dimension} is the following.

\begin{corollary}
    Let \((M,g_M)\) be an $m$-dimensional complete noncompact Riemannian manifold that can be partitioned by a closed hypersurface with nonzero \(\widehat{A}\)-genus. If the scalar curvature satisfies \(\scalM \ge \kappa\) for some constant \(\kappa\le 0\), then
\[
    \lambda_1(M,g_M) \le -\,\frac{m-1}{4m}\,\kappa.
\]
\end{corollary}
                                                                          
Example~\ref{example:mthm_codimension_one} shows that this inequality is sharp.


\subsection{A quantitative version of Anghel's theorem}
We now present the following quantitative strengthening of Anghel's
theorem~\cite{An93}.

\begin{mthm}\label{mthm:quantitative_Anghel}
Let $(M,g_M)$ be an $m$-dimensional complete noncompact Riemannian spin manifold that is partitioned by a closed hypersurface $N$ with nonzero $\widehat{A}$-genus.  
Suppose that $\scalM > 0$ on a compact subset $K \subset M$ containing $N$.
\begin{enumerate}
    \item  
    If $\lambda_1(M,g_M)=0$, then
    \[
        \inf_M \scalM <0.
    \] 
    \item
    If $\lambda_1(M,g_M)>0$, then
    \[
        \inf_M \scalM \leq - \frac{4m}{m-1} \lambda_1(M,g_M).
    \]
\end{enumerate}
\end{mthm}

In other words, the presence of a separating hypersurface with nonzero
$\widehat{A}$-genus obstructs the persistence of positive scalar curvature:
the local positivity on $K$ forces $\scalM$ to become negative somewhere
on $M$.

For a cylinder we obtain the following consequence, which relates to a
conjecture of Rosenberg-Stolz \cite{RS94}*{Conjecture~7.1}.

\begin{corollary}\label{thm:Stolz_Rosenberg_cylinder_case}
Let $N$ be a closed spin manifold with nonzero $\widehat{A}$-genus. Then $M = N\times\mathbf{R}$ does not admit a complete metric with nonnegative scalar curvature that is strictly positive on a neighbourhood of $N$. In particular, $M$ carries no complete metric of positive scalar curvature.
\end{corollary}

\begin{remark}
Cecchini~\cite{Cec20} proved a related result for closed spin manifolds
with nonvanishing Rosenberg index.
\end{remark}

We now establish a splitting theorem for complete noncompact manifolds
with nonnegative scalar curvature.

\begin{theorem}\label{thm:splitting_theorem}
Let \((M,g_M)\) be an \(m\)-dimensional complete noncompact Riemannian spin manifold with nonnegative scalar curvature. Suppose that \(M\) can be partitioned by a hypersurface with nonzero \(\widehat{A}\)-genus. Then \(M\) is isometric to a product \((N\times\mathbf{R},\,g_N+dt^{2})\), where \((N,g_N)\) is a closed Ricci-flat spin manifold.
\end{theorem}


\subsection{Complete metrics of nonnegative scalar and mean curvature}

We begin with the boundary analogue of Theorem~\ref{mthm:quantitative_Anghel}.

\begin{mthm}\label{mthm:boundary_Anghel}
Let $(M,g_M)$ be an $m$-dimensional complete noncompact Riemannian spin
manifold with compact mean-convex boundary, and suppose that $M$ can be
partitioned by a closed hypersurface $N$ with nonzero
$\widehat{A}$-genus. Assume that $\scalM > 0$ on a compact subset
$K \subset M$ containing $N$. Then
\[
\inf_M \scalM < 0.
\]
\end{mthm}

\begin{corollary}\label{cor:boundary_vertical}
Let $M$ be an $m$-dimensional noncompact spin manifold with compact
boundary, and suppose that $M$ can be partitioned by a closed
hypersurface with nonzero $\widehat{A}$-genus. Then $M$ admits no
complete metric with positive scalar curvature and mean-convex boundary.
\end{corollary}

This extends Anghel's result~\cite{An93}*{Theorem~2.1} to manifolds with boundary.

\begin{corollary}\label{cor:noncompact_band_obstruction}
Let \(M\) be an \(m\)-dimensional noncompact spin manifold with compact boundary. Suppose that some component of \(\partial M\) has nonzero \(\widehat{A}\)-genus. Then \(M\) admits no complete metric with positive scalar curvature and mean-convex boundary. 
\end{corollary}

The next result is a rigidity theorem for manifolds with nonnegative
scalar curvature and mean-convex boundary.

\begin{theorem}\label{thm:scalar_mean_rigidity_all_dimensions}
Let \((M,g_M)\) be an \(m\)-dimensional complete noncompact Riemannian spin manifold with nonnegative scalar curvature and compact mean-convex boundary \(\partial M\). If \(M\) can be partitioned by a hypersurface with nonzero \(\widehat{A}\)-genus, then \(\partial M\) must be connected and \(M\) is isometric to a product \( (\partial M \times [0,\infty), g_{\partial M} + dt^2) \), where \((\partial M, g_{\partial M})\) is Ricci-flat.
\end{theorem}

This theorem is a partial converse to the obstruction results: if the
manifold can be partitioned by a hypersurface with nonzero
$\widehat{A}$-genus, then the only way to have nonnegative scalar
curvature is to be a rigid product with a Ricci-flat boundary. The proof
relies on Corollary~\ref{cor:boundary_vertical}.

As an application, we obtain an obstruction for the connected sum of an
$\widehat{A}$-band with any noncompact spin manifold.

\begin{corollary}
Let $M$ be an $m$-dimensional noncompact spin manifold, possibly with
compact boundary, and let $N$ be an $(m-1)$-dimensional closed spin
manifold with nonzero $\widehat{A}$-genus. Then the connected sum
$(N \times [0,1])\# M$\footnote{For manifolds with boundary, the
connected sum is taken in the interior.} cannot admit a complete metric
with nonnegative scalar curvature and mean-convex boundary.
\end{corollary}


\subsection{Calabi-Yau type theorem in higher dimensions}

Recall that a complete Riemannian manifold $(M,g_M)$ has \emph{sublinear volume growth} if
\[
    \liminf_{r\to +\infty}
    \frac{\operatorname{vol}_{g_M}(B_r(p))}{r}=0
\]
for some point $p\in M$. This condition is fundamental because,
by the classical results of Calabi~\cite{Calabi75} and Yau~\cite{Yau76}, a
complete manifold with nonnegative Ricci curvature has sublinear volume
growth if and only if it is compact.

Recently, Zhu~\cite{Zhu24+} and Bi-Zhu~\cite{BZ26+} established a scalar
curvature analogue in low dimensions: for complete spin manifolds of
dimension $2\le m\le 5$ with nonnegative scalar curvature outside a compact
set that are aspherical at infinity, sublinear volume growth forces
compactness. In higher dimensions, however, no comparable
Calabi-Yau type theorem under a nonnegative scalar curvature hypothesis was
previously known.

In this subsection we prove such a theorem in higher dimensions, replacing
asphericity at infinity by a different topological condition involving the
$\widehat{A}$-genus.

\begin{definition}
A complete manifold \(M\) is said to be \emph{partitioned at infinity by a closed hypersurface with nonzero \(\widehat{A}\)-genus} if for every compact subset \(K\subset M\) there exists a compact subset \(L\supset K\) such that \(M\setminus L\) contains a separating closed hypersurface \(\Sigma\) with \(\widehat{A}(\Sigma)\neq 0\).
\end{definition}

The following lemma, due to Bi-Zhu~\cite{BZ26+}, provides a mean-concave
exhaustion for complete manifolds with sublinear volume growth; it was
proved earlier in dimensions up to seven by Lott~\cite{Lott25} and
Zhu~\cite{Zhu24+}.

\begin{lemma}[\cite{BZ26+}*{Theorem~1.2}]\label{lem:zhu}
Let \((M,g_M)\) be an \(m\)-dimensional complete Riemannian manifold with sublinear volume growth.
Then for any bounded region \(K\subset M\) there exists a bounded region \(\Omega\supset K\) such that
\(\partial\Omega\) is mean-concave with respect to the outer unit normal of \(\Omega\).
\end{lemma}

\begin{mthm}\label{mthm:higher_calabi-yau}
Let \((M,g_M)\) be an \(m\)-dimensional complete spin manifold with nonnegative scalar curvature outside a compact subset.  Suppose that \(M\) is partitioned at infinity by a closed hypersurface with nonzero \(\widehat{A}\)-genus.  Then \(M\) has sublinear volume growth if and only if it is compact.
\end{mthm}

\vspace{.3mm}

\textbf{Notation.}
Unless otherwise specified, all manifolds are assumed to be smooth, connected, oriented and of dimension at least two.
Additionally, the notion of completeness refers to metric completeness when \(M\) has nonempty boundary.

\textbf{Organization of the paper.}
In Section~\ref{sec:Callias_operator_spectral_flow}, we provide the necessary technical tools concerning Callias operators and spectral flow. 
In Section~\ref{sec:codimension_one}, we prove Theorem~\ref{mthm:codimension_one}.
Section~\ref{sec:spectral_Anghel} is devoted to the proof of Theorem~\ref{mthm:quantitative_Anghel},
and Section~\ref{sec:boundary_Anghel_Calabi_Yau} contains the proofs of Theorem~\ref{mthm:boundary_Anghel}
and Theorem~\ref{mthm:higher_calabi-yau}.


\section{Preliminaries}\label{sec:Callias_operator_spectral_flow}

This section reviews the fundamental properties of deformed Dirac operators (also called Callias operators).
Let $(M,g_M)$ be an $m$-dimensional complete Riemannian spin manifold, possibly noncompact and with compact boundary.

\subsection{Callias operators and spectral flow}\label{subsec:Callias}

Following \cite{CZ24}*{Example~2.6}, let $\slashed{S}_M$ be the complex spinor bundle on $M$ with the connection induced by the Levi-Civita connection.
Let $\mathcal{E}$ be a Hermitian vector bundle on $M$ with a metric connection. Then the $\mathbf{Z}_2$-graded bundle 
\[
  S:=S^+\oplus S^-:=(\slashed{S}_M \,\otimes\, \mathcal{E}) \oplus (\slashed{S}_M \,\otimes\, \mathcal{E})
\]
is a \textit{relative Dirac bundle} in the sense of \cite{CZ24}*{Definition 2.2} with Clifford multiplication
\[
c:=\begin{pmatrix}
    0 & c_{\slashed{S}} \otimes {\rm id}_{\mathcal{E}} \\
    c_{\slashed{S}} \otimes {\rm id}_{\mathcal{E}} & 0
\end{pmatrix}
\]
and an odd, self‑adjoint, parallel bundle involution
\[
\sigma:=
\begin{pmatrix}
0 & -\sqrt{-1} \\
\sqrt{-1} & 0
\end{pmatrix}
\]
globally defined on $M$.
The Dirac operator $\mathcal{D}$ on $S$ is given by
\[
\mathcal{D}:=\begin{pmatrix}
    0 & \slashed{D}_{\mathcal{E}} \\
   \slashed{D}_{\mathcal{E}}  & 0
\end{pmatrix},
\]
where $\slashed{D}_{\mathcal{E}}$ is the spinor Dirac operator on $M$ twisted with the bundle $\mathcal{E}$.

A Lipschitz function $f: M\to \mathbf{R}$ is called an \textit{admissible potential} if $f$ is equal to a nonzero constant on each component outside a compact subset of $M$ (cf.~\cite{CZ24}*{Definition 3.1}).

\begin{definition}
The \textit{Callias operator} on $S$ associated to the above data is given by
\[
\mathcal{B}_{f} := \mathcal{D} + f \sigma.
\]  
\end{definition}

For the analysis of Callias operators on a manifold $M$ with compact boundary, one must impose appropriate local boundary conditions. The relevant notion is that of a boundary chirality.
\begin{definition}[\cite{CZ24}]
Let $S$ be a relative Dirac bundle and let $s\colon \partial M\to \{\pm 1\}$ be a locally constant function. The \textit{boundary chirality} on $S$ associated to the choice of signs $s$ is the endomorphism
\[
     \chi := s\, c(\nu^*)\sigma \colon \left. S\right|_{\partial M} \to \left. S\right|_{\partial M},
\]
where $\nu^*$ is the dual covector of the outward unit normal $\nu$ to $\partial M$.
\end{definition}

The map $\chi$ is a self‑adjoint, even involution; it anti‑commutes with \(c(\nu^*)\) and commutes with $c(w^*)$ for all $w\in T(\partial M)$.
This leads to the following boundary condition.
\begin{definition}[\cite{CZ24}]
A section $u\in C^{\infty}(M,S)$ satisfies the \textit{local boundary condition} if
\[
  \chi \left(u|_{\partial M}\right) = u|_{\partial M}.
\]
\end{definition}

For a choice of $s:\partial M\to \{\pm 1\}$, denote by \(\mathcal{B}_{f,s}\) the restriction of \(\mathcal{B}_{f}\) to the domain
\[
   {\rm dom}(\mathcal{B}_{f,s}):=\{ u\in C_c^{\infty}(M,S)\colon \chi(u|_{\partial M})=u|_{\partial M} \}.
\]

Note that, by definition, $S=S^+\oplus S^-$ is $\mathbf{Z}_2$-graded and 
$\mathcal{B}_f$ can be decomposed as $\mathcal{B}_f=\mathcal{B}_f^+ \oplus \mathcal{B}_f^-$ where $\mathcal{B}_f^{\pm}$ are differential operators $C^{\infty}(M,S^{\pm}) \to C^{\infty}(M, S^{\mp})$. 
Because \(\chi\) is even with respect to the grading, \(\mathcal{B}_{f,s}\) splits as \(\mathcal{B}_{f,s}=\mathcal{B}_{f,s}^{+}\oplus\mathcal{B}_{f,s}^{-}\) with
\[
   \mathcal{B}_{f,s}^{\pm}: \{u\in C_c^{\infty}(M,S^{\pm}) \colon \chi(u|_{\partial M}) = u|_{\partial M}\} \to L^2(M,S^{\mp}).
\]
By \cite{CZ24}*{Theorem~3.4}, the operator $\mathcal{B}_{f,s}$ is self-adjoint and Fredholm.
Its \textit{index} is defined as
\[
   \ind(\mathcal{B}_{f,s}):=\dim (\ker(\mathcal{B}_{f,s}^+)) - \dim (\ker(\mathcal{B}_{f,s}^-)).
\]

Let $\rho_M\in C^\infty(M,U(l))$ be a smooth map on $M$ with values in the unitary group $U(l)$ such that the commutator \([\mathcal{D},\rho_M]\) defines a bounded operator on \(\operatorname{dom}(\mathcal{B}_{f,s})\).  

Following \cite{Ge93}*{Section~1, p.~491}, consider the trivial bundle \(\mathcal{E}_0 := M \times \mathbf{C}^l\) of rank $l$ over $M$ with a trivial connection $\D$.
The map \(\rho_M\) determines a family of Hermitian connections
\[
    \nabla^{\mathcal{E}_0}(t):=\D + t \rho_M^{-1}[\D, \rho_M], \quad t\in [0,1],
\]
with curvature
\[
    R^{\nabla^{\mathcal{E}_0}(t)}= -t(1-t)(\rho_M^{-1}(\D\rho_M))^2.
\]

Recall that the curvature $R^{\rho_M}$ is defined in \eqref{eq:defn_rho_curvature}.
Note that
\[
    \| R^{\nabla^{\mathcal{E}_0}(t)} \|_{\infty} \leq \|R^{\rho_M}\|_{\infty}
\]
for each $t\in [0,1]$.

Using this family of connections $\{\nabla^{\mathcal{E}_0}(t)\}_{t\in [0,1]}$ we construct a corresponding family of Callias operators. On the twisted Dirac bundle \(S\otimes\mathbf{C}^{l}\) set
\[
\nabla(t) := \nabla^{S}\otimes\operatorname{id} + \operatorname{id}\otimes\nabla^{\mathcal{E}_0}(t), \quad t\in[0,1],
\]
and let \(\mathcal{D}(t)\) be the Dirac operator associated with \(\nabla(t)\); explicitly,
\[
\mathcal{D}(t) = \mathcal{D} + t\,\rho_M^{-1}[\mathcal{D},\rho_M] =(1-t)\mathcal{D}+ t\rho_M^{-1} \mathcal{D} \rho_M.
\]

The involution \(\sigma\) extends naturally to \(S\otimes\mathbf{C}^{l}\); we denote the extension again by \(\sigma\). For notational simplicity we identify \(S\otimes\mathbf{C}^{l}\) with \(S\) when no confusion arises.

Observe that \(\rho_M\) preserves \(\operatorname{dom}(\mathcal{B}_{f,s})\) and commutes with both \(f\) and \(\sigma\). For $t\in[0,1]$, define a family of Callias operators on \(\operatorname{dom}(\mathcal{B}_{f,s})\) by
\[
\mathcal{B}_{f}(t) := (1-t)\mathcal{B}_{f} + t\,\rho_M^{-1}\mathcal{B}_{f}\rho_M
= \mathcal{D}(t) + f\sigma.
\]
Then \(\{\mathcal{B}_{f,s}(t)\}_{t\in[0,1]}\) forms a continuous family of self‑adjoint Fredholm operators in the Riesz topology (cf.~\cite{BMJ05}, \cite{Lesch05}).

\begin{definition}[{\cite{Shi24+}}]
The \textit{spectral flow} of the family \(\{\mathcal{B}_{f,s}(t)\}_{t\in[0,1]}\), denoted by $\sflow(\mathcal{B}_{f,s},\rho_M)$, is defined to be the net number of eigenvalues of $\mathcal{B}_{f,s}(t)$ that change from negative to nonnegative as $t$ increases from $0$ to $1$.
\end{definition}

\begin{remark}
    If every single operator $\mathcal{B}_{f,s}(t)$ in such a family is invertible, there cannot be any eigenvalue changing its sign when $t$ varies from $0$ to $1$, and therefore, the spectral flow $\sflow(\mathcal{B}_{f,s},\rho_M)$ is forced to vanish.
\end{remark}

\subsection{Spectral estimates}
To make the paper self-contained, we reproduce the spectral estimate for a family of Callias operators from \cite{Liu26a+}.
Let \(\{e_i\}_{i=1}^m\) be a local orthonormal frame of \(TM\) and \(\{e^i\}_{i=1}^m\) its dual coframe.
For the Dirac operator \(\mathcal{D}(t)\) where $t\in [0,1]$, the Bochner-Lichnerowicz-Schr\"{o}dinger-Weitzenb\"ock formula (cf.~\cite{LM89}) takes the form
\begin{equation}\label{eq:BLW}
    \mathcal{D}^2(t) = \nabla^*\nabla(t) + \mathscr{R}(t),
\end{equation}
with the connection Laplacian
\[
\nabla^*\nabla(t) := -\sum_i\bigl(\nabla_{e_i}(t)\nabla_{e_i}(t)-\nabla_{\nabla^{TM}_{e_i}e_i}(t)\bigr)
\]
and the curvature endomorphism
\begin{equation}\label{eq:curvature_endomorphism}
    \mathscr{R}(t) := \sum_{i<j} c(e^i)c(e^j)\,R^{\nabla(t)}(e_i,e_j),
\end{equation}
where \(R^{\nabla(t)}\) is the curvature tensor of \(\nabla(t)\).

By Green's formula (cf.~\cite{Taylor11}*{Proposition~9.1}), for any \(u\in C_c^{\infty}(M,S)\) we have
\begin{equation}\label{eq:Green_formula_Dirac}
    \int_M \langle \mathcal{D}(t) u, u\rangle\,dV = \int_M \langle u, \mathcal{D}(t) u\rangle\,dV - \int_{\partial M} \langle u, c(\nu^*) u\rangle\,dA,
\end{equation}
and similarly,
\begin{equation}\label{eq:Green_formula_connection}
    \int_M |\nabla(t) u|^2\,dV = \int_M \langle u,\nabla^*\nabla(t) u\rangle\,dV + \int_{\partial M} \langle u, \nabla_{\nu}(t) u\rangle\,dA.
\end{equation}

Combining \eqref{eq:BLW}, \eqref{eq:Green_formula_Dirac} and \eqref{eq:Green_formula_connection} gives, for every \(u\in C_c^{\infty}(M,S)\),
\begin{equation}\label{eq:BLW+Green}
\begin{aligned}
    \int_M |\mathcal{D}(t)u|^2\,dV = &\int_M|\nabla(t) u|^2\,dV + \int_M \langle u, \mathscr{R}(t) u\rangle\,dV \\
      &- \int_{\partial M} \langle u, c(\nu^*) \mathcal{D}(t) u + \nabla_{\nu}(t) u \rangle\,dA.
\end{aligned}
\end{equation}

Now observe that
\[
|\mathcal{B}_f(t) u|^2 = |\mathcal{D}(t) u|^2 + \langle\mathcal{D}(t)u, f\sigma u\rangle + \langle f\sigma u, \mathcal{D}(t) u\rangle + f^2|u|^2,
\]
which together with \eqref{eq:BLW+Green} yields
\begin{equation}\label{eq:integral_Callias_square}
\begin{aligned}
\int_M |\mathcal{B}_f(t) u|^2\,dV 
= & \int_M |\mathcal{D}(t) u|^2\,dV + \int_M \langle u, \underbrace{(\mathcal{D}(t) f\sigma + f\sigma \mathcal{D}(t))}_{=c({\rm d}f)\sigma} u\rangle\,dV + \int_M f^2|u|^2\,dV \\
 & - \int_{\partial M} \langle u, c(\nu^*) f\sigma u\rangle\,dA. 
\end{aligned}
\end{equation}

On the restricted bundle \(S^{\partial}:=S|_{\partial M}\), introduce a boundary Clifford multiplication and a boundary connection by
\[
   c^{\partial}(e^i):=-c(e^i)c(\nu^*),\qquad \nabla^{\partial}_{e_i}(t):=\nabla_{e_i}(t)-\frac{1}{2}c^{\partial}(\nabla_{e_i}(t)\nu^*),
\]
where now \(\{e_i\}_{i=1}^{m-1}\) is a local orthonormal frame of \(T(\partial M)\) with dual coframe \(\{e^i\}_{i=1}^{m-1}\), and \(\nu^*\) is the dual covector of the outward unit normal \(\nu=e_m\) to \(\partial M\).
The associated family of boundary Dirac operators is
\[
\mathcal{A}(t):=\sum_{i=1}^{m-1} c^{\partial}(e^i)\nabla^{\partial}_{e_i}(t).
\]
A direct computation shows
\[
\mathcal{A}(t)c(\nu^{*})=-c(\nu^{*})\mathcal{A}(t),\quad
\mathcal{A}(t)\sigma=\sigma\mathcal{A}(t),\quad
\chi\mathcal{A}(t)=-\mathcal{A}(t)\chi.
\]
Consequently, for any \(u\in C^{\infty}(M,S)\) satisfying the local boundary condition,
\begin{equation}\label{eq:vanishing_boudary_term}
    \langle u|_{\partial M}, \mathcal{A}(t) u|_{\partial M} \rangle = 0.
\end{equation}

Let \(H:=\frac{1}{m-1}\sum_{i=1}^{m-1}\langle e_i,\nabla_{e_i}\nu\rangle\) be the mean curvature of \(\partial M\) with respect to \(\nu\).
The basic boundary identity (cf.~\cite{Bar96}*{Proposition~2.2}) states
\begin{equation}\label{eq:boundary_identity}
  \mathcal{A}(t)=\frac{m-1}{2}\,H + c(\nu^{*})\mathcal{D}(t) + \nabla_{\nu}(t).
\end{equation}

For \(t\in[0,1]\) define the Penrose operators on \(S\) by
\[
\mathcal{P}_{\xi}(t)u:=\nabla_{\xi}(t)u-\frac1m c(\xi^*)\mathcal{D}(t)u,\qquad \xi\in TM.
\]
They satisfy (cf.~\cite{BHMMM15}*{Section~5.2})
\begin{equation}\label{eq:Penrose_identity}
    |\nabla(t) u|^2 = |\mathcal{P}(t)u|^2 +\frac{1}{m}|\mathcal{D}(t) u|^2.
\end{equation}
Inserting \eqref{eq:BLW+Green}, \eqref{eq:integral_Callias_square}, \eqref{eq:vanishing_boudary_term}, \eqref{eq:boundary_identity} and \eqref{eq:Penrose_identity} one obtains the integral identity
\begin{equation}\label{eq:Penrose_spectral_estimates}
\begin{aligned}
   \int_M  |\mathcal{B}_{f,s}(t) u|^2\,dV
    =& \frac{m}{m-1} \int_M \Big( |\mathcal{P}(t) u|^2 + \langle u, \mathscr{R}(t) u \rangle \Big) dV \\
     & + \int_M  \big\langle u, f^2 u +  c({\rm d} f) \sigma u \big\rangle\,dV - \int_{\partial M} \Big( sf - \frac{m}{2} H \Big) |u|^2\,dA.
\end{aligned}
\end{equation}

We now assume that \(f\) is an admissible potential and that \(u\) is a nonzero spinor with \(\mathcal{B}_{f,s}(t)u = 0\) for some \(t\in[0,1]\).
To estimate the Penrose term from below we follow~\cite{HKKZ24} and set
\(v_i(t):=c(e^i)\nabla_{e_i}(t)u + \frac{f}{m}\sigma u\).
These quantities satisfy \(\sum_{i=1}^m v_i(t)=0\). Writing \(v(t)=(v_1(t),\bar v(t))\) and applying Cauchy-Schwarz to \(\sum_{i=2}^m v_i(t)=-v_1(t)\) yields
\((m-1)|\bar v(t)|^2\ge|v_1(t)|^2\).
Since \(|\mathcal{P}(t)u|^2=\sum_{i=1}^m|v_i(t)|^2=|v_1(t)|^2+|\bar v(t)|^2\), we obtain the pointwise estimate
\begin{equation}\label{eq:Penrose_Kato}
|\mathcal{P}(t)u|^2 \ge \frac{m}{m-1}|v_1(t)|^2
    = \frac{m}{m-1}\Bigl|\nabla_{e_1}(t)u - \frac{f}{m}c(e^1)\sigma u\Bigr|^2.
\end{equation}

Choose a constant \(\gamma>0\) with \(\alpha:=\frac{m}{m-1}-\frac{1}{4\gamma}>0\).
From \eqref{eq:Penrose_Kato} we expand
\[
\begin{aligned}
|\mathcal{P}(t)u|^2
&\ge \frac{m}{m-1}|\nabla_{e_1}(t)u|^2
   -\frac{2}{m-1}f\langle\nabla_{e_1}(t)u,\,c(e^1)\sigma u\rangle
   +\frac{1}{m(m-1)}f^2|u|^2.
\end{aligned}
\]
Completing the square in the first two terms on the right gives
\[
\begin{aligned}
|\mathcal{P}(t)u|^2 = &
\frac{1}{4\gamma}|\nabla_{e_1}(t)u|^2 
+ \alpha\Bigl|\nabla_{e_1}(t)u - \frac{1}{\alpha(m-1)}f\,c(e^1)\sigma u\Bigr|^2 \\
&+ \underbrace{\Bigl(\frac{1}{m(m-1)} - \frac{1}{\alpha(m-1)^2}\Bigr)}_{=:\alpha_1} f^2|u|^2.
\end{aligned}
\]
Dropping the nonnegative squared term yields
\[
|\mathcal{P}(t)u|^2 \ge \frac{1}{4\gamma}|\nabla_{e_1}(t)u|^2 + \alpha_1 f^2|u|^2.
\]

Kato's inequality states \(|\nabla(t)u|\ge |\nabla|u||\) almost everywhere.
At points where \(|u|\) is differentiable we pick a local orthonormal frame with \(e_1\) along the gradient direction, so that \(|\nabla_{e_1}|u||=|\nabla|u||\) (if \(\nabla|u|=0\) the inequality is trivial).
Thus for every \(t\in[0,1]\) we have the pointwise bound
\begin{equation}\label{eq:Kato_Penrose}
    |\mathcal{P}(t) u|^2 \ge \frac{1}{4\gamma} |{\rm d}|u||^2 + \alpha_1 f^2 |u|^2
\end{equation}
almost everywhere on \(M\).

Another application of Green's formula provides the identity (cf.~\cite{HKKZ24}*{(3.12)})
\begin{equation}\label{eq:identity_for_spectral_estimates}
\int_{\partial M} \langle c(\nu^*) u, f\sigma u\rangle\,dA = -\int_M \bigl(2f^2|u|^2 + \langle u, c(\mathrm{d}f)\sigma u\rangle\bigr)\,dV.
\end{equation}

Combining \eqref{eq:Penrose_spectral_estimates}, \eqref{eq:Kato_Penrose} and \eqref{eq:identity_for_spectral_estimates}, for any \(u\in\ker(\mathcal{B}_{f,s}(t))\) with \(t\in[0,1]\) we finally obtain
\begin{equation}\label{eq:spectral_flow_connection_spectral_estimate}
\begin{aligned}
\int_{\partial M} \Big( \alpha_2 sf - \frac{m}{2} H \Big) |u|^2\,dA
\ge & \frac{m}{m-1} \int_M \Big( \frac{1}{4\gamma} |{\rm d}|u||^2 + \langle u, \mathscr{R}(t) u \rangle \Big) dV \\
     & + \int_M  \big\langle u, \alpha_2 f^2 u +  \alpha_2 c({\rm d}f) \sigma u \big\rangle\,dV,
\end{aligned}
\end{equation}
where \(\alpha_2:=1-\frac{m\alpha_1}{m-1}\).
Because \(\alpha>0\), we have \(\alpha_1=\frac{1}{m(m-1)}-\frac{1}{\alpha(m-1)^2}<\frac{1}{m(m-1)}\), and therefore
\[
 \alpha_2 = 1-\frac{m\alpha_1}{m-1} > 1-\frac{1}{(m-1)^2} \geq 0.
\]
An analogous spectral estimate holds for a single Callias operator:
for any \(u\in\ker(\mathcal{B}_{f,s})\),
\begin{equation}\label{eq:even_connection_spectral_estimate}
\begin{aligned}
\int_{\partial M}\Bigl(\alpha_2 sf-\frac{m}{2}H\Bigr)|u|^2\,dA
&\ge \frac{m}{m-1}\int_M\!\Bigl(\frac1{4\gamma}|\mathrm{d}|u||^2+\langle u,\mathscr{R}u\rangle\Bigr)dV \\
&\quad + \int_M\bigl\langle u,\;\alpha_2 f^2 u+\alpha_2 c(\mathrm{d}f)\sigma u\bigr\rangle\,dV,
\end{aligned}
\end{equation}
where \(\mathscr{R}\) is defined in \eqref{eq:curvature_endomorphism}.


\section{\texorpdfstring{Proof of Theorem~\ref{mthm:codimension_one}}{~}}\label{sec:codimension_one}

In this section we prove Theorem~\ref{mthm:codimension_one}. We begin with an
estimate involving an auxiliary parameter \(\gamma\in (\frac{m-1}{4m},\infty)\).

\begin{proposition}\label{pro:spectrl_scalar_vertical}
    Let $(M,g_M)$ be an $m$-dimensional complete noncompact Riemannian spin manifold with at least two ends. 
    If $\gamma\in (\frac{m-1}{4m},\infty)$, then
    \begin{equation}\label{eq:spectrl_scalar_vertical}
         \inf_M \scalM + \gamma^{-1} \lambda_1(M,g_M) \leq \frac{2m(m-1)}{\Avcw(M)}.
    \end{equation}
\end{proposition}

\begin{proof}

It suffices to show that for any partition $(M,N)$ of $M$,
\[
  \inf_M \scalM + \gamma^{-1} \lambda_1(M,g_M) \leq \frac{2m(m-1)}{\Avcw(M,N)}.
\]

We treat odd and even dimensions separately.

\textbf{Case~1}: Assume that $m$ is odd.
If there exist no vertical $\widehat{A}$-admissible bundles for $(M,N)$, then $\Avcw(M, N)=0$ and there is nothing to show. Thus, let $\mathcal{E}$ be a vertical $\widehat{A}$-admissible bundle for $(M,N)$ satisfying
\begin{equation}\label{eq:relative_topo_obst}
    \int_N \Af(N) \wedge \ch(\mathcal{E}|_N) \neq 0.
\end{equation}

Let $S$ be a relative Dirac bundle over $M$ associated with $\mathcal{E}$ and let $\mathcal{D}$ be the corresponding Dirac operator on $S$.

Let $\psi:M\to [0,1]$ be a smooth function such that $\psi=1$ on $M_+$ and $\psi=-1$ outside some compact subset in the interior of $M_-$. 
For $\varepsilon>0$, let $f=\varepsilon \psi$. Then $f$ is an admissible potential on $M$ such that $f=\varepsilon$ on $M_+$, $f=-\varepsilon$ outside some compact subset in the interior of $M_-$. 
Let $V$ be a compact subset of $M$ such that $\partial V=N\cup N_-$, where $N_-$ is contained completely in the portion where $f=-\varepsilon$.

Consider the Callias operator $\mathcal{B}_{f}= \mathcal{D}+f\sigma$.
By \cite{An93}*{Corollary~1.9} (see also \cite{Rad94}), together with
\eqref{eq:relative_topo_obst} and the cohomological index formula
(cf.~\cite{AS63}; see also \cite{LM89}*{p.~256}), we obtain
\[
    \ind(\mathcal{B}_{f}) \neq 0 .
\]
Hence there exists a nonzero $u\in\ker(\mathcal{B}_{f})$. Because $M$ has
no boundary, the boundary term in
\eqref{eq:even_connection_spectral_estimate} disappears. From the spectral
estimate \eqref{eq:even_connection_spectral_estimate} and
\cite{CZ24}*{(2.20)} we deduce
\[
\begin{aligned}
    0 \geq & \frac{m}{m-1} \int_M
            \Bigl( \frac{1}{4\gamma}\,|\mathrm{d}|u||^2
                 + \frac{1}{4}\scalM\,|u|^2
                 + \langle u,\mathscr{R}^{\mathcal{E}}u\rangle \Bigr) \,dV \\
          & + \int_M \bigl\langle u,\,
                \alpha_2 f^2 u + \alpha_2\,c(\mathrm{d}f)\sigma u
                \bigr\rangle \,dV ,
\end{aligned}
\]
where $\gamma>\frac{m-1}{4m}$ and $\alpha_2>0$ are constants, and
\[
\mathscr{R}^{\mathcal{E}}=\sum_{i<j}c(e^i)c(e^j)({\rm id}_{\slashed{S}_M\oplus \slashed{S}_M} \otimes R^{\mathcal{E}}(e_i,e_j))
\]
denotes the curvature endomorphism. Note that $\mathscr{R}^{\mathcal{E}}$ depends linearly on the curvature tensor of $\mathcal{E}$, and that
\begin{equation}\label{eq:curvature_term_bound_index_case}
    \langle u, \mathscr{R}^{\mathcal{E}} u\rangle \geq - \frac{m(m-1)}{2} \|R^{\mathcal{E}}\|_{\infty} \, |u|^2 \quad \text{and}\quad \left \langle u, c({\rm d} f) \sigma u \right \rangle \geq - |{\rm d}f| \, |u|^2, 
\end{equation}
hence
\begin{align*}
    0 \geq & \frac{m}{4\gamma(m-1)} \int_M |\mathrm{d}|u||^2 \,dV \\
          & + \frac{m}{m-1} \int_M
            \Bigl( \frac{1}{4}\scalM\,|u|^2
                 - \frac{m(m-1)}{2}\,\|R^{\mathcal{E}}\|_{\infty}\,|u|^2
            \Bigr) \,dV \\
          & + \alpha_2 \int_M (f^2-|\mathrm{d}f|)\,|u|^2 \,dV .
\end{align*}

Recall that $f=\varepsilon\psi$ and $\supp(\mathrm{d}\psi)\subset V$.
Using the variational characterization
\eqref{defn:bottom_spectrum} we find
\begin{align*}
    0 \geq \Bigl( & \frac{m}{4\gamma(m-1)}\,\lambda_1(M,g_M)
                 + \frac{m}{4(m-1)}\,\inf_M\scalM \\
                & - \frac{m^2}{2}\,\|R^{\mathcal{E}}\|_{\infty}
                  - \varepsilon\alpha_2\sup_V|\mathrm{d}\psi| \Bigr)
            \|u\|_{L^2(M)}^2 ,
\end{align*}
with $\|u\|_{L^2(M)}^2 = \int_M|u|^2\,dV$. Since $\|u\|_{L^2(M)}^2>0$,
dividing by this quantity gives
\[
    \inf_M \scalM + \gamma^{-1}\lambda_1(M,g_M)
    \leq 2m(m-1)\,\|R^{\mathcal{E}}\|_{\infty}
      + \frac{4(m-1)\varepsilon\alpha_2}{m}\,\sup_V|\mathrm{d}\psi| .
\]
By taking the limit as $\varepsilon\to 0$, we obtain
\[
    \inf_M \scalM + \gamma^{-1}\lambda_1(M,g_M)
    \leq 2m(m-1)\,\|R^{\mathcal{E}}\|_{\infty} .
\]
Taking the infimum over all vertical $\widehat{A}$-admissible bundles for
$(M,N)$ produces the desired estimate
\[
    \inf_M \scalM + \gamma^{-1}\lambda_1(M,g_M)
    \leq \frac{2m(m-1)}{\Avcw(M,N)} .
\]
Therefore, we finish the proof of Case~1.

\textbf{Case~2}:
Assume that $m$ is even.
If no vertical $\widehat{A}$-admissible pair exists for $(M,N)$, then
$\Avcw(M,N)=0$ and the inequality is trivial. 
Thus, let $(\mathcal{E}, \rho_{M})$ be a vertical $\widehat{A}$-admissible pair for $(M,N)$ satisfying
\begin{equation}\label{eq:relative_topo_obst_odd}
  \int_N \Af(N) \wedge \ch\left( \left. \mathcal{E} \right|_N \right) \wedge \ch\left( \left. \rho_M \right|_N \right) \neq 0.
\end{equation}

Build the relative Dirac bundle $S$ from $\mathcal{E}$ and denote the
Dirac operator by $\mathcal{D}$. Let $f$ be the same admissible potential function
as in Case~1 and set $\mathcal{B}_f = \mathcal{D}+f\sigma$. 
We then consider the family of Callias operators
\[
  \mathcal{B}_f(t) = (1-t)\mathcal{B}_f + t\, \rho_M^{-1} \mathcal{B}_f \rho_M, \quad t \in [0,1].
\]
By \cite{Shi25+}*{Theorem~1.3}, together with \eqref{eq:relative_topo_obst_odd} and the cohomological formula for the spectral flow (cf. \cite{Ge93}*{Theorem~2.8}), we obtain
\[
  \sflow(\mathcal{B}_f, \rho_M) \neq 0.
\]
Consequently, there exist $t_0\in[0,1]$ and a nonzero
$u\in\ker(\mathcal{B}_f(t_0))$. Because $\partial M=\emptyset$, the
boundary term in \eqref{eq:spectral_flow_connection_spectral_estimate}
vanishes. Applying that estimate together with \cite{Shi25+}*{(4.1)} gives
\[
\begin{aligned}
    0 \geq & \frac{m}{m-1} \int_M
            \Bigl( \frac{1}{4\gamma}\,|\mathrm{d}|u||^2
                 + \frac{1}{4}\scalM\,|u|^2 \\
          & \qquad + \bigl\langle u,\,
                \mathscr{R}^{\mathcal{E}}u
                - 4t(1-t)\mathscr{R}^{\rho_M}u \bigr\rangle \Bigr) \,dV \\
          & + \int_M \bigl\langle u,\,
                \alpha_2 f^2 u + \alpha_2\,c(\mathrm{d}f)\sigma u
                \bigr\rangle \,dV ,
\end{aligned}
\]
where $\gamma>\frac{m-1}{4m}$ and $\alpha_2>0$ are constants, and $\mathscr{R}^{\rho_M}=\sum_{i<j}c(e^i)c(e^j)R^{\rho_M}(e_i,e_j)$.

Using the fact that
\[
    \langle u, \mathscr{R}^{\mathcal{E}} u - 4t(1-t) \mathscr{R}^{\rho_M} u \rangle \geq  - \frac{m(m-1)}{2} (\|R^{\mathcal{E}}\|_{\infty} + \|R^{\rho_M}\|_{\infty}) |u|^2,
\]
and
\[
 \langle u, c({\rm d} f) \sigma u \rangle \geq - |{\rm d}f| \, |u|^2,
\]
we have
\begin{align*}
     0 \geq & \frac{m}{4\gamma(m-1)} \int_M |{\rm d} |u||^2 dV \\
     & + \frac{m}{m-1} \int_M \left( \frac{1}{4} \scalM |u|^2 -\frac{m(m-1)}{2} (\|R^{\mathcal{E}}\|_{\infty} + \|R^{\rho_M}\|_{\infty}) |u|^2 \right) dV \\
    & + \alpha_2 \int_M  (f^2-|{\rm d}f|)|u|^2 dV.
\end{align*}

Recall that $f=\varepsilon \psi$ and note that $\supp({\rm d}\psi)\subset V$. Using \eqref{defn:bottom_spectrum}, we obtain
\begin{align*}
    0 \geq & \Big( \frac{m}{4\gamma(m-1)} \lambda_1(M,g_M) + \frac{m}{4(m-1)} \inf_M \scalM \\
     & \quad - \frac{m^2}{2} (\|R^{\mathcal{E}}\|_{\infty} + \|R^{\rho_M}\|_{\infty}) - \alpha_2 \varepsilon \sup_{V} |{\rm d} \psi| \Big)  \|u\|_{L^2(M)}^2.
\end{align*}
Since $\|u \|_{L^2(M)}^2>0$, we have
\[
  \inf_M \scalM + \gamma^{-1} \lambda_1(M,g_M) \leq 2m(m-1)(\|R^{\mathcal{E}} \|_{\infty} + \|R^{\rho_M}\|_{\infty}) + \frac{4(m-1)\alpha_2 \varepsilon}{m} \sup_{V} |{\rm d} \psi|.
\]
We let $\varepsilon \to 0$ to find that
\[
  \inf_M \scalM + \gamma^{-1} \lambda_1(M,g_M) \leq 2m(m-1) (\|R^{\mathcal{E}} \|_{\infty} + \|R^{\rho_M}\|_{\infty}).
\]
Taking the infimum over all vertical $\widehat{A}$-admissible pairs for $(M,N)$ yields
\[
    \inf_M \scalM + \gamma^{-1}\lambda_1(M,g_M)
    \leq \frac{2m(m-1)}{\Avcw(M,N)} .
\]
Thus, we arrive at the result for Case~2.

Combining the two cases finishes the proof of the proposition.
\end{proof}

\begin{remark}
    For an $m$-dimensional connected oriented (not necessarily complete)
    Riemannian manifold $(M,g_M)$, the \textit{$\gamma$-spectral constant}
    introduced in~\cite{HKKZ24} is defined, for any $\gamma\in\mathbf{R}$, by
    \[
        \Lambda_{\gamma}(g_M)=
        \inf\Bigl\{ \int_M \bigl(|\mathrm{d}u|^2 + \gamma\,\scalM\,u^2\bigr)\,dV
            \colon u\in H_0^1(M),\ \int_M u^2\,dV =1 \Bigr\},
    \]
    where $H_0^1(M)$ is the completion of $C_c^\infty(M)$ in the Sobolev
    $H^1$-norm. Under the assumptions of
    Proposition~\ref{pro:spectrl_scalar_vertical}, the proof above also
    gives
    \[
        \gamma^{-1}\,\Lambda_{\gamma}(g_M) \leq \frac{2m(m-1)}{\Avcw(M)},
        \quad \text{for } \gamma\in \big(\frac{m-1}{4m},\infty\big).
    \]
\end{remark}

We are now in a position to prove Theorem~\ref{mthm:codimension_one}
by using Proposition~\ref{pro:spectrl_scalar_vertical}.

\begin{proof}[Proof of Theorem~\ref{mthm:codimension_one}]
Proposition~\ref{pro:spectrl_scalar_vertical} implies that for every
$\epsilon>0$,
\[
    \inf_M \scalM + \Bigl(\frac{m-1}{4m}+\epsilon\Bigr)^{-1}\lambda_1(M,g_M)
    \le \frac{2m(m-1)}{\Avcw(M)} .
\]
Letting $\epsilon\to0^+$ yields the desired inequality
\[
    \inf_M \scalM + \frac{4m}{m-1}\,\lambda_1(M,g_M)
    \le \frac{2m(m-1)}{\Avcw(M)} .
\]
\end{proof}


\section{\texorpdfstring{Proof of Theorem~\ref{mthm:quantitative_Anghel}}{~}}\label{sec:spectral_Anghel}

We start with a proposition needed for the proof.

\begin{proposition}\label{pro:spectral_Anghel}
Let $(M,g_M)$ be an $m$-dimensional complete noncompact Riemannian spin manifold
that can be partitioned by a closed hypersurface $N$ with nonzero
$\widehat{A}$-genus.
Suppose that $\scalM > 0$ on a compact subset $K \subset M$ containing $N$.
If $\gamma\in (\frac{m-1}{4m},\infty)$, then
\[
 \inf_{M} \scalM < -\frac{1}{\gamma}\,\lambda_1(M,g_M).
\]
\end{proposition}

\begin{proof}[Proof of Proposition~\ref{pro:spectral_Anghel}]
Suppose, by contradiction, that 
\begin{equation}\label{eq:assumption_quantitative_Anghel}                                            
    \inf_{M} \scalM \geq -\frac{1}{\gamma} \lambda_1(M,g_M),
\end{equation}
where $\gamma>\frac{m-1}{4m}$.

By hypothesis, $N$ partitions $M$ into $M_+$ and $M_-$ with
$N = M_+ \cap M_-$. Since $\widehat{A}(N)\neq 0$, the trivial line bundle
$\mathcal{E}=M\times\mathbf{C}$ with the trivial connection is vertical
$\widehat{A}$-admissible. We construct the relative Dirac bundle $S$
using this $\mathcal{E}$ and denote by $\mathcal{D}$ the corresponding
Dirac operator on $S$ (cf.~Section~\ref{subsec:Callias}). Because the
connection on $\mathcal{E}$ is trivial, its curvature vanishes, and the
curvature endomorphism in the
Bochner-Lichnerowicz-Schr\"{o}dinger-Weitzenb\"{o}ck formula
(cf.~\cite{LM89}) reduces to
\begin{equation}\label{eq:BLW-trivial}
    \mathscr{R}= \frac{1}{4}\scalM .
\end{equation}

Since $\scalM>0$ on the compact set $K$, there exists $\sigma>0$ such that
$\scalM\ge \sigma$ on $K$.   
Let $\mathcal{U}$ be a collar neighborhood of $K$ in $M$ with boundary $\partial \overline{\mathcal{U}}=\partial_+\overline{\mathcal{U}} \sqcup \partial_- \overline{\mathcal{U}}$ such that $\scalM\geq \sigma_0$ on $\overline{\mathcal{U}}\setminus K$ for some constant $0<\sigma_0<\sigma$.

Let $\psi:M\to [0,1]$ be a smooth function such that $\psi=0$ on $K$ and $\psi=\pm 1$ on $M_{\pm}\setminus \overline{\mathcal{U}}$. 
For $\varepsilon>0$, let $f=\varepsilon \psi$. Then $f$ is an admissible potential on $M$ such that $f=0$ on $K$, $f=\pm \varepsilon$ on $M_{\pm}\setminus \overline{\mathcal{U}}$.

Consider the Callias operator $\mathcal{B}_{f}=\mathcal{D}+f\sigma$.
By \cite{An93}*{Corollary~1.9} (see also \cite{Rad94}) together with the
cohomological index formula (cf.~\cite{AS63}; see also \cite{LM89}*{p.~256}),
we obtain
\[
    \ind(\mathcal{B}_{f}) = \widehat{A}(N) \neq 0 .
\]
Hence there exists a nonzero $u\in\ker(\mathcal{B}_{f})$. As $M$ has no
boundary, the boundary term in \eqref{eq:even_connection_spectral_estimate}
vanishes. Using the spectral estimate
\eqref{eq:even_connection_spectral_estimate} and \eqref{eq:BLW-trivial} we get
\[
\begin{aligned}
    0 \geq & \frac{m}{4(m-1)} \int_M
            \Bigl( \frac{1}{\gamma}\,|\mathrm{d}|u||^2
                 + \scalM\,|u|^2 \Bigr) \,dV \\
          & + \int_M \bigl\langle u,\,
                \alpha_2 f^2 u + \alpha_2\,c(\mathrm{d}f)\sigma u
                \bigr\rangle \,dV ,
\end{aligned}
\]
where $\gamma>\frac{m-1}{4m}$ and $\alpha_2>0$ are constants.

Splitting the integrals over $K$, $\overline{\mathcal{U}}\setminus K$ and
$M\setminus\overline{\mathcal{U}}$ gives
\[
\begin{aligned}
0 \geq \; & \frac{m}{4\gamma(m-1)}\,\lambda_1(M,g_M)\,\|u\|_{L^2(M)}^2
          + \frac{m}{4(m-1)} \int_K \scalM\,|u|^2 \,dV \\
          & + \Bigl(\frac{m\sigma_0}{4(m-1)}
                   - \varepsilon\alpha_2\sup_{\overline{\mathcal{U}}\setminus K}
                     |\mathrm{d}\psi|\Bigr)
            \|u\|_{L^2(\overline{\mathcal{U}}\setminus K)}^2 \\
          & + \int_{M\setminus\overline{\mathcal{U}}}
            \Bigl(\frac{m\,\scalM}{4(m-1)} + \varepsilon^2\alpha_2\Bigr)
            |u|^2\,dV ,
\end{aligned}
\]
where $\|u\|_{L^2(A)}^2 = \int_A|u|^2\,dV$ for $A\subset M$.
On $K$ we have $\scalM\ge\sigma$, hence
\[
\begin{aligned}
0 \geq \; & \frac{m}{4\gamma(m-1)}\,\lambda_1(M,g_M)\,\|u\|_{L^2(M)}^2
          + \frac{m\sigma}{4(m-1)}\,\|u\|_{L^2(K)}^2 \\
          & + \Bigl(\frac{m\sigma_0}{4(m-1)}
                   - \varepsilon\alpha_2\sup_{\overline{\mathcal{U}}\setminus K}
                     |\mathrm{d}\psi|\Bigr)
            \|u\|_{L^2(\overline{\mathcal{U}}\setminus K)}^2 \\
          & + \Bigl(\frac{m}{4(m-1)}\inf_{M\setminus\overline{\mathcal{U}}}\scalM
                   + \varepsilon^2\alpha_2\Bigr)
            \|u\|_{L^2(M\setminus\overline{\mathcal{U}})}^2 .
\end{aligned}
\]

Choose $\varepsilon>0$ so small that
$\frac{m\sigma_0}{4(m-1)} > \varepsilon\alpha_2\sup_{\overline{\mathcal{U}}\setminus K}|\mathrm{d}\psi|$.
Splitting the first term according to
$M = \mathcal{U} \cup (M\setminus\overline{\mathcal{U}})$ yields
\[
\begin{aligned}
0 \geq \; & \frac{m}{4\gamma(m-1)}\,\lambda_1(M,g_M)\,\|u\|_{L^2(\mathcal{U})}^2
          + \frac{m\sigma}{4(m-1)}\,\|u\|_{L^2(K)}^2 \\
          & + \underbrace{\Bigl(\frac{m\sigma_0}{4(m-1)}
                   - \varepsilon\alpha_2\sup_{\overline{\mathcal{U}}\setminus K}
                     |\mathrm{d}\psi|\Bigr)}_{>0}
            \|u\|_{L^2(\overline{\mathcal{U}}\setminus K)}^2 \\
          & + \underbrace{\Bigl(\frac{m}{4(m-1)}\inf_{M\setminus\overline{\mathcal{U}}}\scalM
                     + \frac{m}{4\gamma(m-1)}\lambda_1(M,g_M)
                     + \varepsilon^2\alpha_2\Bigr)}_{>0\ \text{by } \eqref{eq:assumption_quantitative_Anghel}}
            \|u\|_{L^2(M\setminus\overline{\mathcal{U}})}^2 .
\end{aligned}
\]
All terms on the right-hand side are nonnegative, and some of them have
strictly positive coefficients. This forces $u$ to vanish on
$K \cup (\overline{\mathcal{U}}\setminus K) \cup (M\setminus\overline{\mathcal{U}}) = M$,
a contradiction.
This completes the proof.
\end{proof}

Using Proposition~\ref{pro:spectral_Anghel}, we now prove Theorem~\ref{mthm:quantitative_Anghel}.

\begin{proof}[Proof of Theorem~\ref{mthm:quantitative_Anghel}]
From Proposition~\ref{pro:spectral_Anghel},
for all $\gamma > \frac{m-1}{4m}$ we have
\begin{equation}\label{eq:temp}
\inf_M \scalM < -\frac{1}{\gamma} \lambda_1(M,g_M).
\end{equation}
Note that $\lambda_1(M,g_M) \ge 0$. We will separate the argument into the following two cases.

\textbf{Case~1}: Assume that $\lambda_1(M,g_M) = 0$. 
Taking any $\gamma_0>\frac{m-1}{4m}$ (for instance $\gamma_0=1$) in
\eqref{eq:temp} gives
\[
  \inf_M \scalM < 0.
\]

\textbf{Case~2}: Assume that $\lambda_1(M,g_M)>0$. 
Since \eqref{eq:temp} holds for every $\gamma > \frac{m-1}{4m}$,
$\inf_M \scalM$ is a lower bound of the set $\bigl\{-\frac{1}{\gamma}\lambda_1(M,g_M)\, \colon\, \gamma > \frac{m-1}{4m}\bigr\}$.
The function $\gamma \mapsto -\frac{1}{\gamma}\lambda_1(M,g_M)$ is continuous and strictly increasing on this interval,
so its infimum is the limit as $\gamma \to (\frac{m-1}{4m})^+$, namely
\[
\inf_{\gamma > \frac{m-1}{4m}} \Big\{ - \frac{1}{\gamma}\lambda_1(M,g_M) \, \colon\, \gamma > \frac{m-1}{4m} \Big\}
= - \frac{4m}{m-1}\lambda_1(M,g_M).
\]
Because $\inf_M \scalM$ is a lower bound, the definition of infimum gives
\[
   \inf_M \scalM \le  - \frac{4m}{m-1} \lambda_1(M,g_M),
\]
which is exactly the required inequality.
\end{proof}

\begin{proof}[Alternative proof of Anghel's theorem]
Assume $M$ carries a complete metric $g_M$ with $\scalM>0$.
Choose a compact set $K\subset M$ containing $N$ (for instance, a closed
tubular neighborhood of $N$). Then $\scalM>0$ on $K$, and
Theorem~\ref{mthm:quantitative_Anghel} forces $\inf_M\scalM<0$
(in either case of the theorem), contradicting $\scalM>0$.
\end{proof}

\begin{proof}[Proof of Theorem~\ref{thm:splitting_theorem}]
Assume $\Ric_{g_M}\not\equiv 0$. Then, by a deformation result of
Kazdan~\cite{Kazdan82}, $M$ admits a complete metric of positive scalar
curvature, contradicting Anghel's theorem (and hence also
Theorem~\ref{mthm:quantitative_Anghel}). Thus $\Ric_{g_M}\equiv 0$.

A noncompact manifold that can be partitioned by a closed hypersurface with
nonzero $\widehat{A}$-genus must have at least two ends. Consequently,
$M$ contains a geodesic line, and the Cheeger-Gromoll splitting theorem
\cite{CG72} implies that $M$ is isometric to a product
$(N\times\mathbf{R},\,g_N+dt^{2})$ with $(N,g_N)$ a closed Ricci-flat
manifold. Because $M$ is spin, $N$ inherits a spin structure.
\end{proof}


\section{\texorpdfstring{Proofs of Theorem~\ref{mthm:boundary_Anghel} and Theorem~\ref{mthm:higher_calabi-yau}}{~}}\label{sec:boundary_Anghel_Calabi_Yau}

We begin with the proof of Theorem~\ref{mthm:boundary_Anghel}, which adapts the
arguments of Proposition~\ref{pro:spectral_Anghel} to the boundary setting.

\begin{proof}[Proof of Theorem~\ref{mthm:boundary_Anghel}]
Suppose, by contradiction, that 
\begin{equation}\label{eq:Anghel_boundary_assumption}
    \scalM \geq 0 \ \text{on } M.
\end{equation}

We keep the notation and the overall strategy of the proof of
Proposition~\ref{pro:spectral_Anghel}. By hypothesis, $N$ separates $M$ into
$M_+$ and $M_-$ with $N = M_+ \cap M_-$. Because $\widehat{A}(N) \neq 0$, the
trivial line bundle $\mathcal{E}=M\times\mathbf{C}$ equipped with the trivial
connection is vertical $\widehat{A}$-admissible. Construct the relative Dirac
bundle $S$ from this $\mathcal{E}$ and let $\mathcal{D}$ be the corresponding
Dirac operator (cf.~Section~\ref{subsec:Callias}). For this choice,
\begin{equation}\label{eq:boundary_BLSW}
    \mathscr{R}= \frac{1}{4}\scalM .
\end{equation}

Let $\mathcal{U}$ be a collar neighborhood of $K$ in $M$ with boundary $\partial \overline{\mathcal{U}}=\partial_+\overline{\mathcal{U}} \sqcup \partial_- \overline{\mathcal{U}}$ such that $\scalM\geq \sigma_0$ on $\overline{\mathcal{U}}\setminus K$ for some constant $0<\sigma_0<\sigma$.

Let $\psi:M\to [0,1]$ be a smooth function such that $\psi=0$ on $K$ and $\psi=\pm 1$ on $M_{\pm}\setminus \overline{\mathcal{U}}$ and near the boundary $\partial M$. 
For $\varepsilon>0$, let $f=\varepsilon \psi$. Then $f$ is an admissible potential on $M$ such that $f=0$ on $K$, $f=\pm \varepsilon$ on $M_{\pm}\setminus \overline{\mathcal{U}}$ and near the boundary $\partial M$.

We may assume that both \(M_+\) and \(M_-\) contain parts of \(\partial M\); otherwise the boundary lies on only one side and the argument is simpler.
Set
\[
\partial_+ M := \partial M \cap M_+, \qquad \partial_- M := \partial M \cap M_-,
\]
which are nonempty and compact. On the relative Dirac bundle $S$ consider the
Callias operator $\mathcal{B}_{f,s} = \mathcal{D} + f\sigma$ with the choice of signs
\[
s = +1 \ \text{on } \partial_+ M, \qquad s = -1 \ \text{on } \partial_- M .
\]

By the same reasoning used in the proof of Proposition~\ref{pro:spectral_Anghel}, we see that 
\[
\ind(\mathcal{B}_{f,s}) \neq 0 .
\]
Hence there exists a nontrivial $u \in \ker(\mathcal{B}_{f,s})$ satisfying the
local boundary condition. From \eqref{eq:even_connection_spectral_estimate}
and \eqref{eq:boundary_BLSW} we obtain
\[
\begin{aligned}
0 \geq &\; \frac{m}{4(m-1)} \int_M
        \Bigl( \frac{1}{\gamma}\,|\mathrm{d}|u||^2 + \scalM\,|u|^2 \Bigr) \,dV \\
        & + \int_M \bigl\langle u,\,
                \alpha_2 f^2 u + \alpha_2\,c(\mathrm{d}f)\sigma u
                \bigr\rangle \,dV \\
        & + \int_{\partial_+ M} \bigl( \alpha_2 f + \tfrac{m}{2}H \bigr) |u|^2 \,dA
          + \int_{\partial_- M} \bigl( -\alpha_2 f + \tfrac{m}{2}H \bigr) |u|^2 \,dA,
\end{aligned}
\]
where \(\gamma>\frac{m-1}{4m}\) and \(\alpha_2>0\) are constants.
On $\partial_+M$ we have
$f = \varepsilon$ and $H \ge 0$, hence $\alpha_2 f + \frac{m}{2}H \ge 0$; on
$\partial_-M$ we have $f = -\varepsilon$ and $H \ge 0$, hence
$-\alpha_2 f + \frac{m}{2}H \ge 0$. Thus both boundary integrals are
nonnegative --- this is exactly where the mean-convexity hypothesis $H \ge 0$
enters.

Combining this with the fact that $\langle u, c({\rm d} f) \sigma u \rangle \geq - |{\rm d}f| \, |u|^2$ yields the estimate
\[
\begin{aligned}
0 \geq & \frac{m}{4\gamma(m-1)} \int_M |\mathrm{d}|u||^2 \,dV
        + \frac{m}{4(m-1)} \int_K \scalM\,|u|^2 \,dV \\
        & + \int_{M\setminus K}
          \Bigl( \frac{m\,\scalM}{4(m-1)} + \alpha_2 f^2
               - \alpha_2|\mathrm{d}f| \Bigr) |u|^2 \,dV .
\end{aligned}
\]
Recall that $f = \varepsilon\psi$ and $\supp(\mathrm{d}\psi) \subset
\overline{\mathcal{U}}\setminus K$. Hence
\begin{equation}\label{eq:Anghel_boundary_spectral_estimate}
\begin{aligned}
0 \geq & \frac{m}{4\gamma(m-1)}\,\|\mathrm{d}|u|\|_{L^2(M)}^2
        + \frac{m\sigma}{4(m-1)}\,\|u\|_{L^2(K)}^2 \\
        & + \Bigl( \frac{m\sigma_0}{4(m-1)}
                 - \varepsilon\alpha_2\sup_{\overline{\mathcal{U}}\setminus K}
                   |\mathrm{d}\psi| \Bigr)
          \|u\|_{L^2(\overline{\mathcal{U}}\setminus K)}^2 \\
        & + \underbrace{\Bigl( \frac{m}{4(m-1)}\inf_{M\setminus\overline{\mathcal{U}}}\scalM
                     + \varepsilon^2\alpha_2 \Bigr)}_{>0\ \text{by } \eqref{eq:Anghel_boundary_assumption}}
          \|u\|_{L^2(M\setminus\overline{\mathcal{U}})}^2 .
\end{aligned}
\end{equation}
Choose $\varepsilon > 0$ so small that
$\frac{m\sigma_0}{4(m-1)} > \varepsilon\alpha_2\sup_{\overline{\mathcal{U}}\setminus K}|\mathrm{d}\psi|$.
All four terms on the right-hand side of
\eqref{eq:Anghel_boundary_spectral_estimate} are nonnegative, and some of
them have strictly positive coefficients. Hence $u$ must vanish on
$K \cup (\overline{\mathcal{U}}\setminus K) \cup (M\setminus\overline{\mathcal{U}}) = M$,
contradicting the fact that $u$ is nontrivial.
\end{proof}

Corollary~\ref{cor:boundary_vertical} follows immediately from
Theorem~\ref{mthm:boundary_Anghel}.

\begin{proof}[Proof of Corollary~\ref{cor:noncompact_band_obstruction}]
By the collar neighborhood theorem, a closed hypersurface with nonzero
$\widehat{A}$-genus partitions $M$ near the boundary component whose
$\widehat{A}$-genus is nonzero. The conclusion then follows from
Corollary~\ref{cor:boundary_vertical}.
\end{proof}

We now prove the rigidity theorem for nonnegative scalar curvature and
mean-convex boundary.

\begin{proof}[Proof of Theorem~\ref{thm:scalar_mean_rigidity_all_dimensions}]
If $\Ric_{g_M}\not\equiv 0$, then by a deformation result of Cruz-Santos
\cite{CS25+}*{Theorem~1.1} the manifold $M$ would carry a complete metric with
positive scalar curvature and mean-convex boundary, contradicting
Corollary~\ref{cor:boundary_vertical}. Thus $\Ric_{g_M}\equiv 0$.

Were $\partial M$ disconnected, a splitting theorem for manifolds with boundary
due to Ichida \cite{Ichida81} (see also Kasue \cite{Kas83}*{Theorem~B} and
Croke-Kleiner \cite{CK92}*{Theorem~1}) would imply that $M$ is compact,
contradicting the noncompactness hypothesis. Hence $\partial M$ is connected.

Since $\Ric_{g_M}\equiv 0$ and $\partial M$ is connected, a boundary splitting
theorem of Kasue \cite{Kas83}*{Theorem~C} (see also Croke-Kleiner
\cite{CK92}*{Theorem~2}) applies. We conclude that $M$ is isometric to the
product $(\partial M \times [0,\infty),\, g_{\partial M} + dt^{2})$, with
$(\partial M, g_{\partial M})$ Ricci‑flat.
\end{proof}

We end this section with the proof of the Calabi-Yau type theorem in higher
dimensions.

\begin{proof}[Proof of Theorem~\ref{mthm:higher_calabi-yau}]
We argue by contradiction. Suppose that $(M,g_M)$ is noncompact and has
sublinear volume growth. Let $K_0\subset M$ be a compact set outside which
the scalar curvature is nonnegative.

By Lemma~\ref{lem:zhu} applied to $K_0$, there exists a bounded region
$\Omega\supset K_0$ such that $\partial\Omega$ is mean-concave with respect to
the outer unit normal of $\Omega$. Consequently, $\partial\Omega$ is mean-convex
when regarded as the boundary of the complement $M\setminus\Omega$.

Since $M$ is partitioned at infinity by a closed hypersurface with nonzero
$\widehat{A}$-genus, the definition applied to the compact set
$\overline{\Omega}$ yields a compact set $L \supset \Omega$ and a separating
closed hypersurface $N \subset M\setminus L$ with $\widehat{A}(N) \neq 0$.

Now set $W:=M\setminus\Omega$. The hypersurface $N$ lies in $M\setminus L$ and
is disjoint from the compact collar $L\setminus\Omega$; hence $N$ also separates $W$.
Thus $W$ is a complete noncompact manifold with
compact mean-convex boundary $\partial W=\partial\Omega$, nonnegative scalar
curvature, and it is separated by a closed hypersurface with nonzero
$\widehat{A}$-genus.

Applying Theorem~\ref{thm:scalar_mean_rigidity_all_dimensions} to $W$ shows
that $W$ is isometric to the product
\[
    (\partial W\times[0,\infty),\, g_{\partial W}+dt^{2}),
\]
where $(\partial W, g_{\partial W})$ is a compact Ricci-flat manifold. In
particular, $W$ has linear volume growth.

Because $\Omega$ is bounded, the volume growth of $M$ coincides with that of
$W$, contradicting the assumption that $M$ has sublinear volume growth.
\end{proof}


\vspace{3mm}
\textbf{Acknowledgements.} 
The author is deeply grateful to Professor Weiping Zhang, Professor Guangxiang Su, Professor Zhenlei Zhang and Professor Bo Liu for their continuous encouragement and support.
This work is partially supported by 
the National Natural Science Foundation of China (12501064), the China Postdoctoral Science Foundation (2025M773075, 2025T002TJ, GZC20252016) and the Nankai Zhide Foundation.

\end{document}